\documentclass[12pt,twoside]{amsart}
\usepackage{amssymb,amsmath,amsthm}
\usepackage{verbatim}
\usepackage{graphicx}
\usepackage{epsfig, enumerate}
\usepackage{color, soul}
\usepackage[all]{xy}

\DeclareMathAlphabet{\mathpzc}{OT1}{pzc}{m}{it}
\newcommand{\N}{{\ensuremath{\mathbb{N}}}}

\def\A{\mathcal A}
\def\K{\mathcal K}

\def\overF{\overline{\mathcal F}}
\newtheorem{theorem}{Theorem}[section]
\newtheorem{lemma}[theorem]{Lemma}
\newtheorem{definition}[theorem]{Definition}
\newtheorem{corollary}[theorem]{Corollary}
\newtheorem{proposition}[theorem]{Proposition}
\newtheorem{remark}[theorem]{Remark}
\newtheorem{example}[theorem]{Example}

\def\d{\displaystyle}

\def\pco{\text{{\it p}{\rm -co}}}
\def\co{{\rm co}}
\def\ep{\varepsilon}
\def\m{\mathpzc{m}}
\def\q{\mathpzc{q}}
\def\ker{{\rm ker\, }}

\def\Ai{{(\mathcal A, \|.\|_{\mathcal A})}}
\def\v{\mathpzc{v}}

\title[]{Operator ideals and approximation properties}

\author{Silvia Lassalle  and Pablo Turco}

\thanks{This project was supported in part by PIP 0624 UBACyT 1-218 and
UBACyT 1-746}

\address{Departamento de Matem\'{a}tica, Universidad de San
Andr\'{e}s, Vito Dumas 284, (B1644BID) Victoria, Buenos Aires,
Argentina, FCEN - UBA  and IMAS - CONICET.}
\email{slassall@dm.uba.ar}

 \address{Departamento de Matem\'{a}tica - Pab I,
 Facultad de Cs. Exactas y Naturales, Universidad de Buenos
Aires, (1428) Buenos Aires, Argentina and IMAS - CONICET}
\email{paturco@dm.uba.ar}


\begin{document}

\begin{abstract}
We use the notion of $\A$-compact sets, which are determined by a
Banach operator ideal $\A$, to show that most classic results of
certain approximation properties and several Banach operator ideals
can be systematically studied under this framework. We say
that a Banach space enjoys the $\A$-approximation property if the
identity map is uniformly approximable on $\A$-compact sets by
finite rank operators. The Grothendieck's classic approximation
property is the $\K$-approximation property for $\K$ the ideal of
compact operators and the $p$-approximation property is obtained as
the $\mathcal N^p$-approximation property for $\mathcal N^p$ the
ideal of right $p$-nuclear operators. We introduce a way to measure
the size of $\A$-compact sets and use it to give a norm  on $\K_\A$,
the ideal of $\A$-compact operators. Most of our results concerning
the operator Banach ideal $\K_\A$ are obtained for right-accessible
ideals $\A$. For instance, we prove that $\K_\A$ is a dual ideal, it
is regular and we characterize its maximal hull. A strong concept of
approximation property, which makes use of the norm defined on
$\K_\A$, is also addressed. Finally, we obtain a generalization of
Schwartz theorem with a revisited $\epsilon$-product.
\end{abstract}

\maketitle
\section*{Introduction}

The Grothendieck's classic approximation property is one of the most
important properties in the theory of Banach spaces. A Banach space
has the approximation property if the identity map can be uniformly
approximated by finite rank operators on compact sets. There are
several reformulations of this property, all of them involving at
least one of the concepts: compact operators or uniform convergence
on compact sets. Ever since Grothendieck's famous
{\it{R\'esum\'e}}~\cite{Gro} and reinforced by the fact that there
are Banach spaces which lack the approximation property (the first
example given by Enflo \cite{Enflo}), important variants of this
property have emerged and were intensively studied. The main
developments on approximation properties can be found in \cite{Cas,
LT} and in the references therein.

The main purpose of this article is to undertake the study of a
general method to understand a wide class of approximation
properties and different ideals of compact operators which can be
equally modeled once the system of compact sets has been chosen. To
this end, we use a refined notion of compactness, given by a Banach
operator ideal $\A$, introduced by Carl and Stephani \cite{CaSt}.
Then, we say that a Banach space has the $\A$-approximation property
if the identity map is uniformly approximated by finite rank
operators on $\A$-compact sets.  Also, the system of $\A$-compact
sets induces in a natural way the class of $\A$-compact operators,
consisting of all the continuous linear operators mapping bounded
sets into $\A$-compact sets. This ideal, which we denote by $\K_\A$,
was introduced an studied in \cite{CaSt}. However, the authors do
not emphasize their study from a geometrical point of view. Here, we
introduce a way to measure the size of $\A$-compact sets and then
use our definition to endow $\K_\A$ with a norm $\|\cdot
\|_{\K_\A}$, under which it is a Banach operator ideal.

The paper is divided in three parts. Fixed a Banach operator ideal
$\A$, we give the basics of $\A$-compact sets, then we study the
ideal of $\A$-compact operators $(\K_\A, \|\cdot\|_{\K_\A})$ and finally
we apply the results obtained to study two natural types of
approximation properties induced by $\A$.  To exemplify many of our
results we  appeal to the concept of $p$-compact sets, $p$-compact
operators and the $p$ and $\kappa_p$-approximation properties (see
definitions below). More precisely, in Section~\ref{sets}, we
examine the class of $\A$-compact sets in a Banach space $E$
($K\subset E$ is relatively $\A$-compact if there exist a Banach
space $X$, an operator $T\in \A(X;E)$ and a compact set $M\subset X$
such that $K\subset T(M)$). We show that the definition can be
reformulated considering only operators in $\A(\ell_1; E)$. Also, we
show  that the class of $p$-compact sets fits in this framework for
the ideal of right $p$-nuclear operators, $\mathcal N^p$. This fact
and the notion of $\A$-null sequences \cite{CaSt}, allow us to solve
a question posed in \cite{DelPin} which was also settled
independently by Oja in her recent work \cite{OJA2}.

Section~\ref{operators} is devoted to the ideal of $\A$-compact
operators with an appropriate norm. Our main results are obtained
for right-accessible ideals $\A$, which include minimal and
injective ideals (see definitions below). When $\A$ is
right-accessible, we prove  that $\K_\A$ is a dual operator ideal,
it is regular and we characterize its maximal hull. Also we show
that coincides with the surjective hull of the minimal kernel of
$\A$. Then, different Banach operator ideals may produce the same
ideal of compact operators. This is the case of $\mathcal N_p^d$,
the dual ideal of the $p$-nuclear operators and $\Pi_p^d$, the dual
ideal of the $p$-summing operators; both produce the ideal of
$p$-compact operators, $\K_p$. As a consequence of our results, we
give a factorization of $\K_p$ in terms of $\Pi_p^d$ and $\K$ the
ideal of compact operators. The ideal $\K_p$, also coincides with
that of $\mathcal N^p$-compact operators and it is not
injective~\cite[Proposition~3.4]{GaLaTur}. However, we show that any
injective Banach operator ideal $\A$ produces an injective ideal
$\K_\A$. We finish the section with some results on $\A^d$-compact
operators, with $\A^d$ the dual ideal of $\A$ and apply our results
to give some examples.

Finally, the last section deals with two types of approximation
properties related with $\A$-compact sets. The $\A$-approximation
property can be seen as a way to weaken the Grothendieck's classic
approximation property. It is  defined by changing the system of
compact sets by the system of $\A$-compact sets. For the particular
case of $\mathcal N^p$, we recover the notion of $p$-approximation
property, which was studied for many authors in the last years, see
for instance \cite{AMR, CK, DOPS, LaTur, SiKa}. We prove that a
Banach space $E$ enjoys the $\A$-approximation property if and only
if the set of finite rank operators from $F$ to $E$ is norm-dense in
$\K_\A(F;E)$, for all Banach spaces $F$. If we take into account the
norm  $\|\cdot \|_{\K_\A}$ instead of the supremum norm, we obtain
the $s\A$-approximation property, which is stronger than the
$\A$-approximation property. In this case, when $\A$ is $\mathcal
N^p$, we recover the $\kappa_p$-approximation property, defined in
\cite{DPS_dens} and studied later in \cite{GaLaTur, LaTur, OJA}. We
study in tandem both types of approximation properties and show that
in general they differ. Also, we address the $\A$-approximation
property in terms of a refined notion of the $\epsilon$-product of
Schwartz. \vskip .7cm

Throughout this paper $E$ and $F$ denote Banach spaces, $E'$  and
$B_E$ denote the topological dual and the closed unit ball of $E$,
respectively. A general Banach operator ideal is denoted by $\Ai$.
When the norm $\|\cdot\|_\A$ is understood we simply write $\A$. We
denote by $\mathcal L, \mathcal F, \overline{\mathcal F}$ and $\K$
the operator ideals of linear bounded, finite rank , approximable
and compact operators, respectively; all considered with the
supremum norm. Often, for $x'\in E'$ and $y \in F$, the 1-rank
operator from $E$ to $F$, $x\mapsto x'(x)y$ is denoted by
$x'{\underline{\otimes}} y$.

Along the manuscript, we use several classic Banach operator ideals
such as the ideal of $p$-nuclear, quasi $p$-nuclear and $p$-summing
operators, $1\le p <\infty$, denoted by $\mathcal N_p,
\mathcal{QN}_p$ and $\Pi_p$, respectively. The basics for these
ideals may be found in \cite{DF}, \cite{djt}, \cite{PIE} or
\cite{RYAN}. Also, to illustrate our results, we appeal to the
ideals of right $p$-nuclear operators $\mathcal N^p$, and
$p$-compact operators $\K_p$. To give a brief description of these
ideals, we need some definitions.

Fix $1\le p <\infty$, a sequence $(x_n)_{n}$ in $E$ is said to be
$p$-summable if $\sum_{n=1}^{\infty} \|x_n\|^p<\infty$ and it is
said to be weakly $p$-summable if $\sum_{n=1}^{\infty} |x'(x_n)|^p <
\infty$, for all $x' \in E'$. As usual, $\ell_p(E)$ and
$\ell_p^w(E)$ denote the spaces of all $p$-summable and weakly
$p$-summable sequences in $E$, respectively. Both are Banach spaces,
the first one considered with the norm
$\|(x_n)_n\|_p=(\sum_{n=1}^{\infty} \|x_n\|^p)^{1/p}$ and the second
one with the norm $\|(x_n)_n\|_p^w = \d\sup_{x' \in B_{E'}}
{\{(\textstyle \sum_{n=1}^{\infty}} |x'(x_n)|^p)^{1/p}\}$. For
$p=\infty$, we have the spaces $c_0(E)$ and $c_0^w(E)$ of all null
and weakly null sequences of $E$, respectively. Both spaces are
endowed with their natural norms.

A mapping $T\in \mathcal L(E;F)$ belongs to the ideal of right
$p$-nuclear operators $\mathcal N^p(E;F)$, if there exist sequences
$(x'_n)_n\in \ell^w_{p'}(E')$ and $(y_n)_n \in \ell_{p}(F)$,
$\frac1p+\frac1{p'}=1$ ($\ell_{p'}=c_0$ if $p=1$), such that
$T=\sum_{n=1}^{\infty} x'_n\underline{\otimes} y_n$. The right
$p$-nuclear norm of $T$ is defined by
$\v^p(T)=\inf\{\|(x'_n)_n\|_{\ell^w_{p'}(E')}
\|(y_n)_n\|_{\ell_p(F)} \colon T=\sum_{n=1}^{\infty}
x'_n\underline{\otimes} y_n\}$.

To describe $p$-compact operators, the notion of $p$-compact sets is
required. Following \cite{SiKa}, we say that a subset $K\subset E$
is relatively $p$-compact, $1\le p\le \infty$,  if there exists a
sequence $(x_n)_n\subset \ell_p(E)$ so that $K\subset \pco\{x_n\}$,
where $\pco\{x_n\}=\{\sum_{n=1}^{\infty} \alpha_n x_n \colon
(\alpha_n)_n \in B_{\ell_{p'}}\}$ is the $p$-convex hull of
$(x_n)_n$ and $\frac1p+\frac1{p'}=1$ ($\ell_{p'}=c_0$ if $p=1$).
With $p=\infty$, we have the relatively compact sets and the
balanced convex hull of $(x_n)_n$, $\co\{x_n\}$. A mapping $T\in
\mathcal L(E;F)$ belongs to the ideal of $p$-compact operators
$\K_p(E;F)$, if it maps bounded sets into relatively $p$-compact
sets and $\kappa_p(T)= \inf\{\|(y_n)\|_p \colon T(B_E)\subset
\pco\{y_n\}\}$ is the $p$-compact norm of $T$.

All the definitions and notation used regarding operator ideals can
be found in the monograph by Defant and Floret \cite{DF}. For
further reading on operator ideals we refer the reader to the books
of Pietsch~\cite{PIE}, of Diestel, Jarchow and Tonge~\cite{djt} and
of Ryan~\cite{RYAN}. For further information on approximation
properties, we refer the reader to the survey by Casazza~\cite{Cas}
and to the book of Lindenstrauss and Tzafriri~\cite{LT}, see also
\cite{dfs} and \cite{RYAN}.

\section{On compact sets and operator ideals} \label{sets}
Fix a Banach operator ideal $\Ai$. Following \cite{CaSt}, a set
$K\subset E$ is relatively $\A$-compact if there exist a Banach
space $X$, an operator $T\in \A(X;E)$ and a compact set $M\subset X$
such that $K\subset T(M)$. A sequence $(x_n)_{n}\subset E$ is
$\A$-convergent to zero if there exist a Banach space $X$ and $T\in
\A(X;E)$ with the following property: given $\ep >0$ there exists
$n_\ep \in \N$ such that $x_n\in \ep T(B_X)$, for all $n\ge n_\ep$.
There is a handy characterization of $\A$-null sequences.
\medskip

\begin{lemma}~\cite[Lemma 1.2]{CaSt}
Let $E$ be a Banach space and $\A$ a Banach operator ideal. A
sequence $(x_n)_{n}\subset E$ is $\A$-null if and only if there
exist a Banach space $X$, an operator $T\in \A(X;E)$ and a null
sequence $(y_n)_n \subset X$ such that $x_n=T(y_n)$, for all $n$.
\end{lemma}

The following characterization of $\A$-compactness was extracted
from~\cite[Section 1]{CaSt}.

\begin{theorem}\label{carlste}
Let $E$ be a Banach space, $K\subset E$ a subset and $\A$ a Banach
operator ideal. The following are equivalent.
\begin{enumerate}[\upshape (i)]
\item $K$ is relatively $\A$-compact.
\item There exist a Banach space $X$ and an operator $T \in
\A(X;E)$ such that for every $\ep >0$ there are finitely many
elements $z_i^{\ep} \in E$, $1\leq i \leq k_{\ep}$ realizing a
covering of $K$: $K\subset \bigcup_{i=1}^{k_{\ep}} \{z_i^{(\ep)} +
\ep T(B_X)\}$.
\item There exists an $\A$-null sequence $(x_n)_n\subset E$
such that $K\subset \co\{x_n\}$.
\end{enumerate}
\end{theorem}

\begin{example}\label{p-compact=A-compact} Compact sets are $\K$-compact sets
and  $p$-compact sets are $\mathcal N^p$-compact sets.
\end{example}

\begin{proof}Let $K\subset E$ be a $p$-compact set, then $K\subset
\pco\{x_n\}=\{\sum_{n=1}^{\infty} \alpha_n x_n \colon (\alpha_n)_n
\in B_{\ell_{p'}}\}$ with $(x_n)_n$ in $\ell_p(E)$  and
$\frac1p+\frac1{p'}=1$ ($\ell_{p'}=c_0$ if $p=1$). Take
$\beta=(\beta_n)_n \in B_{c_0}$ such that $(\frac{x_n}{\beta_n})_n
\in \ell_p(E)$. Let $y_n=\frac{x_n}{\beta_n}$ and $(e'_n)_n$  the
sequence of coordinate functionals on $\ell_{p'}$. Define $T\colon
\ell_{p'}\to E$ the linear operator by $T=\sum_{n=1}^\infty
e_n'\underline{\otimes} y_n$. Then, $K\subset T(M)$ with
$M=\{(\alpha_n\beta_n)_n\colon (\alpha_n)_n\in B_{\ell_{p'}}\}$. The
result follows by noting that $T\in \mathcal N^p (\ell_{p'};E)$ and
$M\subset B_{\ell_{p'}}$ is relatively compact.
\end{proof}

Recently, Delgado and Pi\~neiro~\cite{DelPin} define $p$-null
sequences, $p\ge 1$,  as follows.  A sequence $(x_n)_n$ in a Banach space $E$ is $p$-null if, given
$\ep > 0$, there exist $n_0 \in \N$ and $(z_k)_k\in \ep B_{\ell_p(E)}$ such that
$x_n\in  \pco\{z_k\}$ for all $n\ge n_0$. In \cite[Theorem~2.5 ]{DelPin}, $p$-compact sets are
characterize  as those which are contained in the convex hull of a $p$-null sequence. Then,
the authors prove, under certain hypothesis
on the Banach space $E$, that a sequence is
$p$-null if and only if it is norm convergent to zero and relatively
$p$-compact,~\cite[Proposition~2.6]{DelPin}. Also, they
wonder if the result remains true for arbitrary Banach spaces. We
give an affirmative answer to this question as an immediate
consequence of the above results.

\begin{corollary}\label{DP_question} Let $E$ be a Banach space and $1\le p \le \infty$. A sequence
$(x_n)_n \subset E$
is $p$-null if and only if $(x_n)_n$ is norm convergent to zero and relatively
$p$-compact.
\end{corollary}

\begin{proof}
By Example~\ref{p-compact=A-compact}, the definition of $p$-null
sequences coincides with that of $\A$-null sequences for
$\A=\mathcal N^p$. Then, the result follows
from~\cite[Lemma~1.2]{CaSt} and the equivalence (i) and (iii) of
Theorem~\ref{carlste}.
\end{proof}

When this manuscript was complete we learned that E. Oja has also obtained
Corollary~\ref{DP_question}. She gives a description of the space of $p$-null sequences as a tensor
product via the Chevet-Saphar tensor norm. As an application, the result is obtained
\cite[Theorem~4.3]{OJA2}.
\medskip

Now, we introduce a way to measure the size of relatively $\mathcal A$-compact sets.
Let $K\subset E$ be a relatively
$\mathcal A$-compact set we
define
$$
\m_{\A}(K;E)=\inf\{\|T\|_{\A} \colon K\subset T(M), \ T\in
\mathcal A(X;E)\ \ M\subset B_X\},
$$
where the infimum is taken considering all Banach spaces $X$, all
operators $T\in \mathcal
A(X;E)$ and all compact sets $M\subset B_X$ for which the inclusion $K\subset
T(M)$ holds.

There are some properties which derive directly from the definition of $\m_\A$.
For instance $\m_\A(K;E)=\m_\A(\overline{\co} \{K\};E)$. Also, since $\|T\|\leq \|T\|_\A$ for
any Banach operator ideal $\A$ and $T\in \A(X;E)$ then,  $\sup_{x\in
K}\|x\|\leq \m_\A(K;E)$. Moreover, if $\mathcal B$ is a Banach operator ideal such that
$\A \subset \mathcal B$,  a set $K\subset E$ is  $\mathcal B$-compact  whenever it is
$\A$-compact and $\m_{\mathcal B}(K;E)\leq \m_\A(K;E)$.

\begin{remark}\label{m_p=m_np} {\rm Example~\ref{p-compact=A-compact} shows
that given a sequence $(x_n)_n \in \ell_p(E)$, there exist an operator $T\in
\mathcal N^p (\ell_{p'};E)$ and a relatively compact set $M\subset
B_{\ell_{p'}}$ such that $\pco\{x_n\}=T(M)$. Moreover, fixed $\ep>0$, we can
choose $T$ such that $\|(x_n)_n\|_p\leq\|T\|_{\mathcal N^p} \leq \|(x_n)_n\|_p
+\ep$. Then, if $K\subset E$ is $p$-compact,
$$
\m_{\, \mathcal N^p}(K;E)=\inf\{\|(x_n)\|_p \colon K\subset \pco\{x_n\}\}.
$$
}
\end{remark}

Note that $\m_{\, \mathcal N^p}$ recovers the size of $p$-compact sets defined
in~\cite[Definition~2.1]{GaLaTur}. Also, note that if  $F$ is a Banach space containing $E$,  any set
$K\subset E$  is $\A$-compact as a set of $F$ whenever it is $\A$-compact as a set of $E$, and
$\m_\A(K;F)\leq \m_\A(K;E)$. However,  the definition of $\m_\A$ may depend on the space the sets
are considered, as it is shown in \cite[Corollary 3.5]{GaLaTur}. We will prove,  with additional
hypotheses on $\A$,  that  the $\m_\A$-size of an $\A$-compact set remains the same, see
Corollary~\ref{m_A no cambia},  for any pair of Banach spaces $E$ and $F$ such that $E\subset F$, and
Corollary~\ref{m_A no cambia con E''} for the special case when $F=E''$.

The next result shows that the definition of $\mathcal A$-compact sets (and therefore the size
$\m_\A$) can be reformulated considering only operators in $\A(\ell_1; E)$.

\begin{proposition}\label{paso por ell1}
Let $E$ be a Banach space, $K\subset E$ a subset and $\A$ a Banach
operator ideal. The following are equivalent.
\begin{enumerate}[\upshape (i)]
\item $K$ is relatively $\mathcal A$-compact.
\item There exist a Banach space $X$, operators $T\in \mathcal A(X;E)$ and
$S\in
\overline{\mathcal F}(\ell_1;X)$ and a relatively compact set $M\subset
B_{\ell_1}$ such that
$K\subset T\circ S (M)$. Moreover,
$$
\m_{\A}(K;E)=\inf\{\|T\|_{\A} \|S\| \colon K\subset T\circ S(M); \ M\subset
B_{\ell_1}\}.
$$
where the infimum is taken over all Banach spaces $X$, operators $T$ and
$S$ and sets $M$ as above.
\end{enumerate}
\end{proposition}
\begin{proof} First note that any set $K$ as in (ii) is relatively $\mathcal
A$-compact. Indeed, suppose that there exists a Banach space $X$
such that $K\subset T\circ S (M)$ for $T \in \mathcal A(X;E)$,
$S\in\overline{\mathcal F}(\ell_1;F)$ and $M\subset B_{\ell_1}$
relatively compact. Then, $K\subset \|S\| T(\frac{S(M)}{\|S\|})$
with $\frac{S(M)}{\|S\|}\subset B_X$ a relatively compact set. Also,
$\m_{\A}(K;E)\leq \|S\| \|T\|_{\A}$.

For the converse, since $K\subset E$ is relatively $\A$-compact,
given $\ep>0$ there exist a Banach space $X$, a compact set
$L\subset B_X$ and a operator $T\in \A(X;E)$ such that $K\subset
T(L)$ and $\|T\|_{\A}\leq \m_{\A}(K;E)+\ep$. Since $L\subset B_X$ is
compact, there exists a sequence $(y_n)_n \in c_0(X)$ such that
$L\subset \{\sum_{n=1}^{\infty} \alpha_n y_n \colon (\alpha_n)_n \in
B_{\ell_1}\}$ and $\sup_{n \in \N}\|y_n\|\leq 1 + \ep$. Choose a
sequence $(\beta_n)_n \in B_{c_0}$ such that
$(\frac{y_n}{\beta_n})_n \in c_0(X)$ and $\sup_{n \in
\N}\|\frac{y_n}{\beta_n}\| \leq \sup_{n \in \N}\|y_n\|+\ep\leq
1+2\ep$. Call $z_n=\frac{y_n}{\beta_n}$ and $M=\{(\gamma_n)_n \in
\ell_1 \colon \gamma_n=\beta_n\alpha_n, \ \ (\alpha_n)_n \in
B_{\ell_1}\}$. Note that $M\subset B_{\ell_1}$ is a relatively
compact set. Define the operator $S\colon \ell_1\rightarrow X$ as
$S(\gamma_n)=\sum_{n=1}^{\infty} \gamma_nz_n$. Since
$S(B_{\ell_1})\subset \co\{z_n\}$ and $\ell_1'$ has the
approximation property, $S$ is approximable and $\|S\|\leq \sup_{n
\in \N}\|z_n\|\leq 1+2\ep$.  Moreover,
$$
S(M)=\{\sum_{n=1}^{\infty} \gamma_n z_n \colon (\gamma_n)_n \in
M\}=\{\sum_{n=1}^{\infty} \alpha_n y_n \colon (\alpha_n)_n \in B_{\ell_1}\},
$$
then $L\subset S(M)$ and
$$
K\subset T(L)\subset T\circ S(M).
$$
We also have
$$
\|T\|_{\A}\|S\|\leq \|T\|_{\A}(1+2\ep)\leq (\m_{\A}(K;E)+\ep)(1+2\ep),
$$
and the result follows by letting $\ep\to 0$.
\end{proof}

\begin{corollary}\label{a-comp aprox} Let $E$ be a Banach space, $K\subset E$ a subset and $\A$ a
Banach operator ideal. Then, $K$ is relatively $\A$-compact if and only if $K$ is relatively $\A
\circ \overline{\mathcal F}$-compact and $\m_{\A}(K;E)=\m_{\A \circ \overline{\mathcal
F}}(K;E)$.
\end{corollary}

\begin{proof}
Since $\mathcal A \circ \overline{\mathcal
F} \subset \A$,  every relatively $\A \circ
\overline{\mathcal
F}$-compact set is relatively $\A$-compact  and $\m_{\A}(K;E)\leq
\m_{\A \circ \overline{\mathcal
F}}(K;E)$. The other implication is given by the item (ii)
of the above
proposition, which also gives that
$$
\begin{array}{rl}
\m_{\A}(K;E) &= \inf\{\|T\|_{\A} \|S\| \colon K\subset T\circ S(M); \ M\subset
B_{\ell_1}\}\\
& \geq
\inf\{\|T\|_{\A} \|S\| \colon K\subset T\circ
S(M); \ M\subset B_{F}\}\\
& = \m_{\A \circ \overline{\mathcal
F}}(K;E)
\end{array}
$$
and the proof is complete.
\end{proof}

\begin{corollary} Let $E$ be a Banach space, $\A$ a Banach
operator ideal and $K\subset E$ a relatively $\A$-compact set. Then,
$$
\m_{\A}(K;E)=\inf\{\|T\|_{\A} \colon K\subset T(M), \ \ T\in \A(\ell_1;E) \ \
\text{and} \ \
M\subset B_{\ell_1}\},
$$
where the infimum is taken over all operators $T\in \A(\ell_1;E) $
and all relatively compact sets $M\subset B_{\ell_1}$ such that
$K\subset T(M)$.
\end{corollary}

\begin{proof} With standard notation, Proposition~\ref{paso por ell1}
gives
$$
\begin{array}{rl}
\m_{\A}(K;E)=&\inf\{\|T\|_{\A} \|S\| \colon K\subset T\circ S(M); \ M\subset
B_{\ell_1}\}\\
\geq& \inf\{\|T\|_{\A} \colon K\subset T(M); \ M\subset B_{\ell_1}\}\\
\geq&\m_{\A}(K;E),
\end{array}
$$
which proofs the result.
\end{proof}

\section{The ideal of $\A$-compact operators}
\label{operators}

Associated to the concept of $\A$-compact sets we have the notion of
$\mathcal A$-compact operators, which generalizes that of compact
operators. An operator $T\in \mathcal L(E;F)$ is said to be
$\A$-compact if $T(B_E)$ is a relatively $\A$-compact set in
$F$,~\cite[Definition~2]{CaSt}. The space of $\A$-compact operators,
denoted by $\mathcal K_{\A}$, becomes an Banach operator ideal if
endowed with the norm defined, for any $T\in \K_{\A}(E;F)$, by
$$
\|T\|_{\K_{\A}}=\m_{\A}(T(B_E);F).
$$

With  Example~\ref{p-compact=A-compact} and Remark~\ref{m_p=m_np} we obtain our first
example.

\begin{example} The ideal $\K_p$ and the ideal of  $\mathcal N^p$-compact
operators coincide isometrically.
\end{example}

We propose to study some properties enjoyed by $\K_\A$ and by the operator ideals obtained by the procedures
$$
\A \to \A^d, \quad \A \to \A^{min},\quad \A \to \A^{max}, \quad \A \to \A^{sur},\quad \A \to
\A^{inj}, \quad \A \to
\A^{reg}.
$$
The definitions of these procedures  will
be given opportunely.  We start by recalling  the dual ideal of
$\A$, $\A^d$. Given $T\in \mathcal L$ denote its adjoint by $T'$, then $\A^d (E; F) =
\{T \in \mathcal L(E;F)\colon T'\in \A(F', E')\}$ and $\|T\|_{\A^d}=\|T'\|_{\A}$. Also, recall that
the minimal kernel of $\A$, $\A^{min}$ is the composition ideal  $\A^{min}= \overF \circ \A \circ
\overF$ considered with its natural norm. The ideal is said to be minimal if $\A=\A^{min}$,
isometrically.

Most of our results are obtained for right-accessible operator
ideals. By~\cite[Proposition~25.2~(2)]{DF} we may consider
right-accessible Banach operator ideals as those which satisfy
$\A^{min} = \A \circ \overline{\mathcal F}$, isometrically.
By~\cite[Corollary~21.3]{DF} the left-accessible Banach operator
ideals as those  satisfying that its dual operator ideal is
right-accessible. Also, $\A$ is totally-accessible if for every
finite rank operator $T \in \mathcal L(E;F)$ and $\ep>0$ there exist
$Y\subset F$ a finite-dimensional subspace, $X\subset E$ a subspace
of finite-codimensional and $S \in \mathcal L(E/X;Y)$ such that $T=I_F
S Q_E$ and $\|S\|_{\mathcal A}\leq (1+\ep)\|T\|_{\mathcal A}$, where
$Q_E\colon E\rightarrow E/X$ and $I_F\colon Y\rightarrow F$ are the
canonical quotient mapping and the inclusion, respectively. If $\A$
is totally-accessible then it is right and left-accessible, see
\cite[21.2]{DF}. Also, any minimal ideal is right and left-accessible,
\cite[Corollary~25.3]{DF}.

\subsection{On $\A$-compact operators related with surjective and injective hulls.}
Recall that an operator $T \in \mathcal L(E;F)$ belongs to the surjective hull of $\A(E;F)$,
$\A^{sur}(E;F)$, if and only if $T\circ \q_E$ belongs to $\A$ where $\q_E\colon
\ell_1(B_E)\twoheadrightarrow  E$ is the canonical surjection and
$\|T\|_{\A^{sur}}=\|T\circ \q_E\|_{\A}$. On the other hand, an operator $T \in \mathcal L(E;F)$
belongs to the injective hull of $\A(E;F)$, $\A^{inj}(E;F)$, if and only if $\iota_F\circ T\in \A$,
where $\iota_F\colon F \hookrightarrow \ell_\infty(B_{F'})$ is the canonical injection and
$\|T\|_{\A^{inj}}=\|\iota_F\circ T\|_\A$. The ideal
$\A$ is surjective if $\A=\A^{sur}$ and it is injective if $\A=\A^{inj}$, isometrically. Any
injective ideal  is right-accessible while any surjective ideal is left-accessible \cite[21.2]{DF}.
\medskip

In~\cite[Theorem~2.1]{CaSt}, the operator ideal $\K_\A$ is described
in terms of $\A^{sur}$ via  the identities: $\K_{\A}=(\A\circ
\K)^{sur}=\A^{sur} \circ \K$. Then, we have two direct consequences.
First, $\K_\A=\K_{\K_\A}$, and the process only may produce a new
operator ideal the first time it is applied. Also,  $\K_\A$ is
surjective. From this second fact we observe that the ideal of
nuclear operators $\mathcal N$ does not coincide with $\K_\A$, for
any Banach operator ideal $\A$. With the next proposition we give a
slight improvement of \cite[Theorem~2.1]{CaSt} by considering
approximable instead of compact operators.

\begin{proposition}\label{a-comp con F}
Let $\A$ be a  Banach operator ideal. Then
$$
\K_{\A}=\K_{\A\circ \overline{\mathcal F}}=(\A\circ
\overline{\mathcal F})^{sur}=\A^{sur} \circ \K, \hskip 1cm
isometrically.
$$
\end{proposition}

\begin{proof} The isometric result is obtained by using the definition of
$\|\cdot \|_{\K_\A}$ along the proof of~\cite[Theorem~2.1]{CaSt}. An application
of Corollary~\ref{a-comp aprox} completes the proof.
\end{proof}

\begin{corollary}\label{a-min sur}
Let $\A$ be a right-accessible Banach operator ideal. Then,
$$
\hskip 3cm \K_{\A}=(\A^{min})^{sur}, \hskip 3cm isometrically.
$$
\end{corollary}

\begin{proposition}\label{tot-acces}
If $\A$ is right-accessible, then $\K_\A$ is totally-accessible. In particular,
$$
\hskip 3cm \K_\A=(\K_\A^{min})^{sur}, \hskip 2.5cm  isometrically.
$$
\end{proposition}

\begin{proof}
By Proposition~\ref{a-comp con F}, $\K_{\A}=\A^{sur} \circ \K$. As $\A$ is right-accessible, by
\cite[Ex 21.1]{DF}, $\A^{sur}$ is totally-accessible. Also, $\K$ is injective and surjective, then
$\K$ is injective and left-accessible, \cite[21.2]{DF}. Hence, by \cite[Proposition~21.4]{DF},
$\K_{\A}$ is totally-accessible.
Now, since $\K_\A=\K_{\K_\A}$ and $\K_\A$ is totally-accessible, a direct application of
Corollary~\ref{a-min sur} gives the result.
\end{proof}

Notice that for $\A$ and $\mathcal B$ two Banach operator ideals
such that $\A^{sur}=\mathcal B^{sur}$, Proposition~\ref{a-comp con F} gives $\K_\A=\K_{\mathcal B}$.
If, in addition,
$\A$ and $\mathcal B$ are right-accessible and $\A^{min}=\mathcal B^{min}$, then we also have
$\K_\A=\K_{\mathcal B}$. Combining this two facts, we obtain different descriptions of the ideal of
$p$-compact operators and a factorization via the dual ideal of the $p$-summing operators. This last
result was recently also obtained by Ain, Lillemets and Oja~\cite[Corollary~4.9]{AiLiOj}
independently.

\begin{example}\label{Factorizacion K_p} Let $1\le p\le \infty$. The following isometric
identities hold.
\begin{enumerate}[\upshape (a)]
\item $\K_p=(\K_p^{min})^{sur}.$
\item $\K_p=\K_{\Pi_p^{d}}=\K_{\mathcal N_p^d}.$
\item $\K_p= \Pi_p^{d}\circ \K,. $
\end{enumerate}
\end{example}

\begin{proof} By  \cite[Proposition~3.9]{GaLaTur}, $\K_p$ is totally-accessible. Then, an
application of  Corollary~\ref{a-min sur} gives  $\K_p=(\K_p^{min})^{sur}$, and (a) is obtained.

To prove that $\K_p=\K_{\Pi_p^{d}}$, note that as a consequence
of~\cite[Proposition~3.9]{GaLaTur}, both ideals $\K_p^{min}$ and
$(\Pi_p^{d})^{min}$ coincide isometrically. Also,  $\Pi_p^{d}$ is
totally-accessible \cite[Remark~3.7]{GaLaTur}. Then, using  (a) and
Corollary~\ref{a-min sur} we have $\K_p=(\Pi_p^d)^{min \
sur}=\K_{\Pi_p^d}$. For the other identity, note that $\mathcal
N^p=(\mathcal N_p^d)^{min}$. Since $\mathcal N_p$ is
left-accessible, $\mathcal N_p^d$ is right-accessible. Another
application of  Corollary~\ref{a-min sur} gives $\K_{\mathcal
N_p^d}=(\mathcal N_p^d)^{min \ sur}=\mathcal N^{p \ sur}=\K_p$, and
(b) is proved.

Statement (c) is a direct application of (b) and Proposition~\ref{a-comp con F}. In fact,
$\Pi_p^d$ is a surjective ideal and therefore $\K_p=\K_{\Pi_p^{d}}=\Pi_p^{d}\circ \K$ holds
isometrically.
\end{proof}

In general, $\K_\A$ is not an injective ideal (consider, for instance, the ideal of
$p$-compact operators  \cite[Proposition~3.4]{GaLaTur}). However, an injective  ideal $\A$ gives an
injective ideal of $\A$-compact operators. To show this, we need a preliminary lemma. Although,  we
believe that it should be a known result, we have not found
it in the literature as stated here and we prefer to include a proof.

\begin{lemma}\label{Aminsur}
Let $\A$ be a Banach operator ideal. Then,
$(\A^{inj \ min})^{sur}= (\A^{min \ inj})^{sur}$, isometrically.
\end{lemma}

\begin{proof} Let $E$ and $F$ be Banach spaces.
An operator $T$ belongs to $(\A^{inj \ min})^{sur}(E;F)$ if and
only if $\q_E \circ T \in \A^{inj\  min}(\ell_1(B_E);F)$,  where $\q_E\colon
\ell_1(B_E)\twoheadrightarrow E$ is the canonical surjection. Since $(\ell_1(B_E))'$ has the
approximation property, by \cite[Proposition~25.11.2]{DF}, $\A^{inj \ min}(\ell_1(B_E);F)=A^{min \
inj}(\ell_1(B_E);F)$ isometrically. Therefore, $\q_E \circ T \in \A^{min \ inj
}(\ell_1(B_E);F)$. Equivalently, $T \in (\A^{min \ inj})^{sur}(E;F)$, which proves the lemma.
\end{proof}

\begin{proposition}
Let $\A$ be an injective Banach operator ideal then, $\K_\A$ is also injective. That is,
$$
\hskip 3cm \K_{\A}=\K_{\A}^{inj}, \hskip 3cm  isometrically.
$$
\end{proposition}

\begin{proof}
Since $\A$ is injective, it is right-accessible.
Applying
Corollary~\ref{a-min sur}, \cite[Proposition~8.5.12]{PIE} and Lemma~\ref{Aminsur}  we get
$$
\K_\A^{inj}=(\A^{min \ sur})^{inj}=(\A^{min \ inj})^{sur}=(\A^{inj \
min})^{sur}=(\A^{min})^{sur}=\K_\A.
$$
All the identities are isometric identifications, thus the proof is
complete.
\end{proof}

Notice that there are non injective Banach operator ideals that may induce an injective ideal
of compact operators, this is the case of $\overline{\mathcal F}$ and $\K_{\overline{\mathcal
F}}=\K$. For operators ideals $\A$ such that $\K_\A$ is injective, we show that a set is
$\A$-compact regardless it is considered as  set of a Banach space $F$ or  as set of a
closed subspace  $E$ of $F$, with equal size.
% For this, recall that an operator ideal
% $\A$ is injective if for each metric injection $\iota\colon E \hookrightarrow F$ a map $T\in
% \mathcal L(E;F)$ belongs to  $\A$ whenever $\iota\circ T$ belongs to $\A$ and
% $\|T\|_{\A}=\|\iota\circ T\|_{\A}$, \cite[9.7]{DF}.
\begin{corollary}\label{m_A no cambia} Let $E$  be a Banach space,
$K\subset E$ a subset and $\A$  a Banach operator ideal such that $\K_\A$ is injective. Then,
$K$ is relatively $\A$-compact in $E$ if and only if $K$ is relatively $\A$-compact in $F$,
for every Banach space $F$ containing $E$. Moreover,  $\m_\A(K;E)=\m_\A(K;F)$.
In particular, the result applies to any injective Banach operator ideal $\A$.
\end{corollary}

\begin{proof}
One implication is clear. For the other one, let $\iota\colon E \hookrightarrow F$ a metric
injection such that $\iota(K)\subset F$ is relatively $\A$-compact. We may assume that $K$ is
convex, balanced and closed. Since,  $K$ is compact, by \cite[Lemma~4.11]{RYAN}, there are a Banach
space $G$ and a operator $T\in \mathcal L(G;E)$ such that $T(B_G)=K$. Then,  $\iota \circ T \in
\K_\A(G;F)$. Since $\K_\A=\K_\A^{inj}$,  $K$ is relatively $\A$-compact in $E$ and
$\m_{\A}(K;E)=\m_{\A}(K;F)$.
\end{proof}

\subsection{On the maximal hull of $\A$-compact operators}
We examine the maximal hull of $\K_\A$ for right-accessible ideals $\A$. The maximal hull of
$\A$, $\A^{max}$ consists of all the operators $T$ such that $R\circ T\circ S \in \A$ for any
approximable operators $R$ and $S$ and $\|T\|_{\A^{max}}=\sup\{\|R\circ T\circ S \|_\A\colon
\|R\|,\|S\|\le1 \}$.  The ideal is maximal if $\A=\A^{max}$, isometrically.

\begin{proposition}\label{max}
Let $\A$ be a right-accessible Banach operator ideal, then
$$
\hskip 3cm  \K_\A^{max}=(\A^{max})^{sur}, \hskip 2cm   isometrically.
$$
In particular, if $\A$ is maximal, then $\K_\A^{max}=\A^{sur}$.
\end{proposition}

\begin{proof} Any Banach operator ideal satisfies the isometric identities: $(\A^{max})^{sur}=(\A^{sur})^{max}$ and $(\A^{min})^{max}=\A^{max}$, see
\cite[Proposition 8.7.14]{PIE} and \cite[Proposition 8.7.15]{PIE}. Now, by Corollary~\ref{a-min sur} we have
$$
\K_\A^{max}=(\A^{min \ sur})^{max}=(\A^{min \ max})^{sur}=(\A^{max})^{sur}. \qedhere
$$
\end{proof}

The following corollary follows from Proposition~\ref{a-comp con F} and Proposition~\ref{max}.

\begin{corollary} If $\A$ is right-accessible and maximal then
$$
\hskip 3cm  \K_\A=\K_\A^{max} \circ \K,  \hskip 2cm  isometrically.
$$
\end{corollary}

When we apply the above results to the ideal of $p$-compact operators, we obtain
\cite[Theorem~11]{PIE2} and \cite[Theorem~12]{PIE2}, see also \cite[Corollary~3.6]{GaLaTur}.

\begin{example} For $1\leq p \leq \infty$, the following isometric identities hold.
\begin{enumerate}[\upshape (a)]
\item $\K_p^{max}=\Pi_p^{d}$.
\item $(\K_p^{d})^{max}=\Pi_p$.
\end{enumerate}
\end{example}

\begin{proof} By \cite[Proposition 8.7.12]{PIE}, $(\A^{d})^{max}=(\A^{max})^{d}$, for any ideal $\A$. Then, statement~(b) follows from (a). To prove (a), note that $\Pi_p^d$ is maximal and surjective, thus we have the result from Example~\ref{Factorizacion K_p} item~(b) and Proposition~\ref{max}.
\end{proof}

\subsection{On the regular hull of $\K_\A$ and dual operator ideals}
The dual ideal of any Banach operator ideal is regular \cite[Ex.22.6]{DF}. Now, we show that the
regular hull of $\K_{\A}$ is a dual operator ideal. Let us recall the definitions. The regular hull
of  $\A$, $\A^{reg}$, is the class of all  $T\in \mathcal L(E;F)$ such that $J_F \circ T \in
\A(E;F'')$  and $\|T\|_{\A^{reg}}=\|J_F \circ T\|_{\A}$,
where $J_F\colon F \rightarrow F''$ is the canonical inclusion.
The operator ideal  $\A$ is regular if $\A=\A^{reg}$, isometrically.

\begin{proposition}\label{reg=dd}
Let $\A$ be a Banach operator ideal. Then,
$$
\hskip 3cm  \K_{\A}^{reg}= \K_{\A}^{dd}, \hskip 3cm  isometrically.
$$
\end{proposition}

\begin{proof} We always have $\K_{\A}^{dd}\subset \K_{\A}^{reg}$,
with $\|\cdot\|_{\K_{\A}^{reg}}\leq \|\cdot\|_{\K_{\A}^{dd}}$.  For the other inclusion, let $E$ and $F$ be Banach spaces and take $T \in \K_{\A}^{reg}(E;F)$. In
particular, $T$ is a compact operator.  Thus,
$$
T''(B_{E''})\subset \overline{J_F\circ  T(B_E)}^{w^*} = \overline{J_F\circ  T(B_E)}.
$$
Then, $T'' \in \K_{\A}(E'';F'')$. Moreover,
$$\|T\|_{\K^{dd}_{\A}}=\|T''\|_{\K_\A}\leq\m_{\A}(\overline{J_F \circ
T(B_E)};F'')=\|J_F\circ T\|_{\K_\A}=\|T\|_{\K^{reg}_{\A}},$$
and the isometric result holds.
\end{proof}

For right-accessible ideals we have the following.
\begin{proposition}\label{A-comp reg}
Let $\A$ be a right-accessible Banach operator ideal, then $\K_\A$ is regular. That is,
$$
\hskip 3cm \K_{\A}=\K_{\A}^{reg}, \hskip 3cm  isometrically.
$$
\end{proposition}

\begin{proof}
We always have $\K_{\A}\subset \K_{\A}^{reg}$ and $\|\cdot\|_{\K_{\A}^{reg}}\le
\|\cdot\|_{\K_{\A}}$. Now,
let $E$ and $F$ be Banach spaces and $T \in \K_{\A}^{reg}(E;F)$, then $J_F
\circ T \in \K_\A(E;F'')$.
Since $\A$ is right-accessible, by Corollary~\ref{a-min sur}, $J_F\circ T\in
(\A^{min})^{sur}(E;F'')$. By \cite[Proposition 9.8]{DF}, if $\q_E\colon
\ell_1(B_E)\twoheadrightarrow E$ is the canonical surjection, then $J_F \circ T\circ \q_E \in
\A^{min}(\ell_1(B_E);F'')$. As the dual of $\ell_1(B_E)$ has the approximation property, by
\cite[Corollary~22.8.2]{DF}, $T\circ \q_E \in \A^{min}(\ell_1(B_E);F)$ and $\|T\circ
\q_E\|_{\A^{min}}=\|J_F \circ T\circ
\q_E\|_{\A^{min}}$. Another
application of \cite[Proposition 9.8]{DF} and Corollary~\ref{a-min sur} gives
that $T\in
(\A^{min})^{sur}(E;F)=  \K_\A(E;F)$ and
$\|T\|_{\K_{\A}}=\|T\|_{\K^{reg}_{\A}}$.
\end{proof}

As a consequence of the above, we prove that whenever $\A$ is right-accessible, $\K_\A$ is a dual
operator ideal.

\begin{proposition}\label{K_A dd}
Let $\A$ be a Banach operator ideal. Then, $\K_{\A}\subset \K_{\A}^{dd}$ and
$\|\cdot\|_{\K^{dd}_{\A}}\leq \|\cdot\|_{\K_{\A}}$. Moreover, if $\A$ is
right-accessible, then
$$
\hskip 3cm \K_{\A}= \K_{\A}^{dd}, \hskip 3cm isometrically.
$$
\end{proposition}

\begin{proof} Since $\K_{\A}\subset \K_{\A}^{reg}$ and $\|\cdot\|_{\K_{\A}^{reg}}\le
\|\cdot\|_{\K_{\A}}$, the first statement holds by Proposition~\ref{reg=dd}.  For a
right-accessible ideal $\A$, an application of Proposition~\ref{A-comp reg} completes the proof.
\end{proof}

For operator ideals $\A$ such that $\K_\A$ is regular we can show that a set $K$ is  $\A$-compact
regardless it is considered as a subset of a Banach space $E$ or as a subset of its bidual $E''$,
with equal size. For $p$-compact sets, this was shown in \cite[Theorem 2.4]{GaLaTur}, see also
\cite[Corollary~3.6]{DPS_adj}.

\begin{corollary}\label{m_A no cambia con E''}
Let $E$  be a Banach space,
$K\subset E$ a subset and  $\A$ be a Banach operator ideal such that $\K_\A$ is regular.  Then, $K$
is relatively $\A$-compact if and only if $K\subset E''$ is relatively
$\A$-compact and $\m_{\A}(K;E)=\m_{\A}(K;E'')$.

In particular, the result applies to any right-accessible Banach operator ideal $\A$.
\end{corollary}

\begin{proof} The result is obtained with a similar proof to that given in Corollary~\ref{m_A no
cambia}.
\end{proof}
\medskip

\subsection{On the ideal of $\A^d$-compact operators}
When we consider a left-accessible ideal $\A$ and inspect the ideal of $\A^d$-compact operator, we
can push a little bit further. We finish this section with two results that we apply to recover
some relations satisfied by the ideals of $p$-compact  and quasi $p$-nuclear
operators. We need a preliminary lemma.

\begin{lemma}\label{Amin_sur_dual}
Let $\A$ be a Banach operator ideal. Then,
$(\A^{d \ min})^{sur}= (\A^{min \ d})^{sur}$,  isometrically.
\end{lemma}

\begin{proof}
The same
proof given in Lemma~\ref{Aminsur} works here using \cite[Corollary~22.8.1]{DF} instead of
\cite[Proposition~25.11.2]{DF}.
\end{proof}

\begin{proposition}\label{K_A=Bd}
Let $\A$ be a left-accessible Banach operator ideal. Then
$$
\K_{\A^d}= (\A^{min \ inj})^d, \hskip 3cm isometrically.
$$
\end{proposition}

\begin{proof}
Since $\A^d$ is right-accessible, we apply Corollary~\ref{a-min
sur}, the above lemma and \cite[Theorem 8.5.9]{PIE} to obtain the
isometric identities
$$
\K_{\A^d}=(\A^{d \ min})^{sur}=(\A^{ min \ d})^{sur}=(\A^{ min \ inj})^{d}. \qedhere
$$
\end{proof}

\begin{proposition}\label{K_Ad}
Let $\A$ be a left-accessible Banach operator ideal. Then
$$
\K_{\A^d}^d= \A^{min \, inj}, \hskip 3cm isometrically.
$$
\end{proposition}

\begin{proof} One inclusion is obtained by Proposition~\ref{K_A=Bd} and the fact that injective
ideals are always regular, then
$$
\K_{\A^d}^d = (\A^{min \ inj})^{dd} \subset (\A^{min \ inj})^{reg} =\A^{min \ inj}.
$$
For the reverse inclusion, notice that
$\A^{min}=(\A^{dd})^{min}\subset (\A^{min})^{dd}$, for any $\A$.
Now, considering the injective hulls and applying \cite[Theorem
8.5.9]{PIE}, we obtain
$$
 \A^{min  \ inj}\subset (\A^{min \ dd})^{inj}\subset ((\A^{min \ d})^{sur})^d.
$$
By Lemma~\ref{Amin_sur_dual} and Corollary~\ref{a-min sur},
$$
((\A^{min \ d})^{sur})^d=((\A^{d \ min})^{sur})^d =\K_{\A^{d}}^d.
$$
All the inclusions considered are given by contractive maps, which
completes the proof.
\end{proof}

We illustrate the above propositions with the following examples. In particular, the third
statement recovers a result of \cite[Corollary 3.2]{SiKa2008}. The other two identifications appear
in \cite{DPS_adj} and \cite{GaLaTur}.

\begin{example} Let $1\leq p< \infty$. The following isometric identities hold.
\begin{enumerate}[\upshape (a)]
\item $\K_p =\mathcal{QN}_p^d$.
\item $\K_p^d =\mathcal{QN}_p$.
\item $\K_p^d =(\Pi_p^{min})^{inj}$.
\end{enumerate}
\end{example}

\begin{proof}
By Example~\ref{Factorizacion K_p}  (a), write $\K_p=\K_{\mathcal N_p^{d}}$. Note that
$\mathcal N_p$ is minimal, hence it is left-accessible. Now we use thta $\mathcal N_p=\mathcal{QN}_p^{inj}$ with  Proposition~\ref{K_A=Bd} to obtain (a) and with Proposition~\ref{K_Ad} to obtain (b). For the proof of (c), by Example~\ref{Factorizacion
K_p} (b),  write $\K_p=\K_{\Pi_p^{d}}$. Now, use that $\Pi_p$ is left-accessible and
Proposition~\ref{K_Ad} .
\end{proof}

\section{Approximation properties given by operator ideals}

In this section we study two different types of approximation
properties defined through $\A$-compact sets. To this end we consider  two
different topologies on $\mathcal L(E;F)$.

\begin{definition} Let $\A$ be an operator ideal. On $\mathcal L(E;F)$, we consider
the topology of uniform convergence on $\mathcal A$-compact sets,
$\tau_{\A}$, which is given by the seminorms
$$
q_{K}(T)=\sup_{x \in K} \|T(x)\|,
$$
where $K$ ranges over all $\A$-compact sets. When $F=\mathbb C$, we
simply write $E'_{\A}=(\mathcal L(E;\mathbb C);\tau_{\A})$.
\end{definition}

Note that if $\A=\K$ we obtain $E'_c$, the dual space of $E$ endowed with the
topology of uniform convergence on compact sets.

The other topology we consider is induced by the size of the $\A$-compact
sets $\m_\A$, defined in Section~\ref{sets}.

\begin{definition} Let $\A$ be an operator ideal. On $\mathcal L(E;F)$, we define
the topology of strong uniform convergence on $\mathcal A$-compact
sets, $\tau_{s\A}$, which is given by the seminorms
$$
q_{K}(T)=\m_{\A}(T(K);F),
$$
where $K$ ranges over all $\A$-compact sets.
\end{definition}

The following statements have straightforward proofs.

\begin{remark}\label{A and B}
Let $\A$ be a Banach operator ideal and let $E$ and $F$ be Banach spaces.
\begin{enumerate}[\upshape (a)]
\item The topologies $\tau_{s\K}$ and $\tau_{\K}$ coincide
on $\mathcal L(E;F)$.
\item The topologies $\tau_{s\A}$ and $\tau_{\A}$ coincide on $\mathcal
L(E;\mathbb C)$.
\item $Id\colon (\mathcal L(E;F),\tau_{s\A}) \rightarrow (\mathcal
L(E;F),\tau_{\A})$ is
continuous.
\item  $Id\colon (\mathcal
L(E;F),\tau_{\mathcal B}) \rightarrow
(\mathcal L(E;F),\tau_{\A})$ is continuous, for any Banach operator ideal  $\mathcal B$ such that
$\A \subset \mathcal B$.
\end{enumerate}
\end{remark}

Based on the Grothendieck's classic approximation property we have the following definitions.

\begin{definition} Let $E$ be a Banach space and $\A$ a Banach operator ideal.

We say that $E$ has the $\A$-approximation property %($\A$-approximation property for short)
if $\mathcal F(E;E)$ is $\tau_{\A}$-dense in $\mathcal L(E;E)$.

Also, $E$ has the (strong) $s\A$-approximation property
%($s\A$-approximation property for short)
if $\mathcal F(E;E)$ is $\tau_{s\A}$-dense in $\mathcal L(E;E)$.
\end{definition}

It is clear that the $\mathcal K$, the $s\mathcal K$ and the classic approximation properties
coincide for any Banach space. Also, the classic approximation property implies the
$\A$-approximation property, for any $\A$. However, it may not imply the
$s\A$-approximation property, see Example~\ref{ejemplosPA} (a) below or the comments below
Proposition~3.10 in \cite{GaLaTur}. From Remark~\ref{A and B} (c),
we see that the $s\A$-approximation property is stronger than the
$\A$-approximation property, although the converse might be false as
Example~\ref{ejemplosPA} (a) below shows. Furthermore, if $\A$ and
$\mathcal B$ are two Banach operator ideals and $\A \subset \mathcal
B$, from Remark~\ref{A and B} (d),  the $\mathcal B$-approximation
property implies the $\A$-approximation property. Nonetheless, a
Banach space may have the $s\mathcal B$-approximation property and
fail to have the $s\A$-approximation property, see
Example~\ref{ejemplosPA} (b).

Notice that $\mathcal N^p$-approximation property is the $p$-approximation property
introduced by Sinha and Karn in~\cite{SiKa} and then studied
in~\cite{AMR, CK, DOPS, LaTur}. On the other hand, the $s\mathcal
N^p$-approximation property coincides with the
$\kappa_p$-approximation property defined by Delgado, Pi\~neiro and
Serrano (see~\cite[Remark~2.2]{DPS_dens}) and studied later in
\cite{GaLaTur, LaTur, OJA}. In many cases, the $s\mathcal N^p$ and
the $\mathcal N^p$-approximation properties coincide. This happens,
for instance, on any closed subspace of $L_p(\mu)$. Moreover, in any
closed subspace of $L_p(\mu)$, the $s\mathcal N^p$-approximation
property coincides with the $\K$-approximation
property~\cite[Theorem~1]{OJA}. However, these  properties may
differ as it is summarized below.

\begin{example}\label{ejemplosPA} Let $1< p < 2$.  There exists a Banach space
\begin{enumerate}[\upshape (a)]
\item  with the $s\mathcal N^p$-approximation property which fails to
have the approximation property.
\item   with the $\mathcal N^p$-approximation property which fails to
have the $s\mathcal N^p$-approximation property.
\item  with the $s\mathcal N^2$-approximation property which fails to
have the $s\mathcal N^p$-approximation property.
\end{enumerate}
\end{example}

\begin{proof} Fix $1< q<2$ and let $X$ be a subspace of $\ell_q$,  without the approximation
property. Note that $X$ is reflexive and has cotype 2. Now, combining the comment below
\cite[Proposition~21.7]{DF} and \cite[Corollary~2.5]{DPS_dens} we see that $X'$ the
$\kappa_p$-approximation property for any $1< p < 2$.  Then,  $E=X'$ is an example satisfying (a).

If $1\leq p <2$, every Banach space has the $\mathcal N^p$-approximation
property~\cite[Theorem~6.4]{SiKa}. But given $1<p<2$ there exists a Banach space which fails to have
the $s\mathcal N^p$-approximation property~\cite[Theorem~2.4]{DPS_dens}. Which proves (b).

Finally, every Banach space has the $s\mathcal N^2$-approximation
property~\cite[Corollary~3.6]{DPS_dens}, but given $1<p<2$ there exists a Banach space which fails
to have the $s\mathcal N^p$-approximation property, which proves (c).
\end{proof}

\subsection{On finite rank and compact operators and approximation properties}
Now, we inspect the $\A$ and the $s\A$-approximation properties in
relation with the ideal of finite rank and compact operators. The
first proposition refines a classical result on approximation
properties and finite rank operators. Its proof is standard and  we
omit it.

\begin{proposition}\label{PAidentity} Let $E$ be a Banach space. The following
are equivalent.
\begin{enumerate}[\upshape (i)]
\item $E$ has the $\A$-approximation property ($s\A$-approximation property) .
\item $Id \in \overline{\mathcal F(E;E)}^{\tau_{\A}}$ ($Id \in
\overline{\mathcal F(E;E)}^{\tau_{s\A}}$).
\item For every Banach space $F$, $\mathcal F(E;F)$ is
$\tau_{\A}$-dense ($\tau_{s\A}$-dense) in
$\mathcal L(E;F)$.
\item For every Banach space $F$, $\mathcal F(F;E)$ is $\tau_{\A}$-dense
($\tau_{s\A}$-dense) in
$\mathcal L(F;E)$.
\end{enumerate}
\end{proposition}

A Banach space $E$ has the approximation property if and only if the space of finite rank operators
from any Banach space $F$ into $E$ is norm dense in the ideal of compact operators. The analogous
result remains valid for the $\A$-approximation property and the $s\A$-approximation property. Here,
the ideal of $\A$-compact operators replaces $\K$.

\begin{proposition}\label{PAyoper1}
Let $E$ be a Banach space and $\A$ a Banach operator ideal. The following are equivalent.
\begin{enumerate}[\upshape (i)]
\item $E$ has the $\A$-approximation property.
\item For all Banach spaces $F$, $\mathcal F \circ \K_{\A}(F;E)$ is
$\|.\|$-dense in $\K_{\A}(F;E)$.
\item For all Banach spaces $F$, $\mathcal F(F;E)$ is $\|.\|$-dense in
$\K_{\A}(F;E)$.
\end{enumerate}
\end{proposition}

\begin{proof}
To prove that (i) implies (ii) take $T \in \K_{\A}(F;E)$ and
$\ep>0$. Since $T(B_F)$ is relatively $\A$-compact and $E$ has the
$\A$-approximation property, by Proposition~\ref{PAidentity} there exists $S\in \mathcal
F(E;E)$ such that $\sup_{\ x \in T(B_F)}\|Sx-x\|\leq \ep$. Then,
$\|ST-T\|\leq \ep$ and (ii) holds.

The inclusion $\mathcal F \circ \K_{\A}\subset \mathcal F$ proves
that (ii) implies (iii). To show that (iii) implies (i) take $K\subset E$ an $\A$-compact
set and $\ep>0$. There exist a Banach space $G$, an absolutely
convex compact set $L\subset G$ and an operator $T\in \A(G;E)$ such
that $K\subset T(L)$. Now, we may find an absolutely convex compact
set $\widetilde{L} \subset G$, a Banach space $F$ and an injective
operator $i\in \mathcal L(F;G)$ such that $L\subset \widetilde{L}$
and $i(B_F)=\widetilde{L}$, see for
instance~\cite[Lemma~4.11]{RYAN}. In particular $i$ is a compact
operator and $i^{-1}(L) \subset B_F$ is compact.

Consider the following diagram
$$
\xymatrix{
F \ar[r]^{i} \ar[dr]_{q}&  G \ar[r]^{T} & E  \\
& F/\ker T\circ i  \ar[ur]_{\overline{T\circ i}}  &
}
$$
where $q$ is the quotient map and $T\circ i={\overline{T\circ i}}
\circ q$. Since $\overline{T\circ i}(B_{ F/\ker T\circ i})=T\circ
i(B_F)=T(\widetilde{L})$ is an $\A$-compact set, then
$\overline{T\circ i} \in \K_{\A}(F/\ker T\circ i;E)$. By hypothesis,
there exists a finite rank operator $S\colon F/\ker T\circ i \to E$,
with $\|S-\overline{T\circ i}\|\leq \ep/2$. Write
$S=\sum_{j=1}^{n} y'_j\underline{\otimes}x_j$ with $x_1,\ldots,x_n
\in E$ and $y'_1,\ldots, y'_n \in (F/\ker T\circ i)'$. To find $R
\colon E\to E$ a finite rank operator approximating the identity on
$K$, note that $q(i^{-1}(L)) \subset F/\ker T\circ i$ is compact.
Then, for $\delta=\ep/2\sum_{j=1}^{n} \|x_j\|$ take
$x'_1,\ldots,x'_n \in E'$ such that
$$
\sup_{w \in
q(i^{-1}(L))}|y'_j(w)-(\overline{T\circ i})'(x'_j)(w)|\leq \delta,
$$
for all $j=1,\ldots,n$, and  define $R$ by
$R=\sum_{j=1}^{n} x'_j\underline{\otimes}x_j$. Then,
$$
\begin{array}{rl}
\sup_{x \in K}\|(Id-R)(x)\|\leq &\sup_{y \in L}\|(Id -R)(Ty)\|\\
                   =&\sup_{y \in L}\|(T\circ i-R\circ T\circ i)(i^{-1}(y))\|\\
                   \leq& \sup_{y \in L}\|(T\circ i-S\circ q)(i^{-1}(y))\| +
\sup_{y \in L}\|(S\circ q-  R\circ T\circ i)(i^{-1}(y))\|.
\end{array}
$$
On one hand we have
$$
\begin{array}{rl}
\sup_{y \in L}\|(T\circ i-S\circ q)(i^{-1}(y))\| \le& \sup_{z \in B_F}\|(T\circ
i-S\circ q)(z)\|\\
\le &\sup_{w \in B_{F/\ker T\circ i}} \|(\overline{T\circ
i}-S)(w)\|\\ =& \|S-\overline{T\circ i}\|.
\end{array}
$$
On the other hand,
$$
\begin{array}{rl}
\sup_{y \in L}\|(S\circ q-  R\circ T\circ i)(i^{-1}(y))\| =& \sup_{y \in
L}\|(\sum_{j=1}^{n}
y'_j\circ q \underline{\otimes}x_j -\sum_{j=1}^{n} x'_j\circ T\circ i
\underline{\otimes}x_j)(i^{-1}(y))\|\\
\le& \sum_{j=1}^{n} \sup_{w \in
q(i^{-1}(L))}|y'_j(w)-\overline{T\circ
i}'(x'_j)(w)|\|x_j\|\\
\leq& \sum_{j=1}^{n} \delta \|x_j\|.
\end{array}
$$
Therefore,  $\sup_{x \in K}\|(Id-R)(x)\| \leq \ep$,
proving (i).
\end{proof}

The analogous result concerning the $s\A$-approximation property requires the norm
$\|.\|_{\K_{\A}}$. Its proof is similar to that given above and we omit it.

\begin{proposition}\label{PAyoper2}
Let $E$ be a Banach space and $\A$ a Banach operator ideal. The following are equivalent.
\begin{enumerate}[\upshape (i)]
\item $E$ has the $s\A$-approximation property.
\item For all Banach spaces $F$, $\mathcal F \circ \K_{\A}(F;E)$ is
$\|.\|_{\K_{\A}}$-dense in $\K_{\A}(F;E)$.
\item For all Banach spaces $F$, $\mathcal
F(F;E)$ is $\|.\|_{\K_{\A}}$-dense in $\K_{\A}(F;E)$.
\end{enumerate}
\end{proposition}

The notion of the minimal kernel of an operator ideal allows us to  give another characterization
of the $s\A$-approximation property.

\begin{corollary}
Let $E$ be a Banach space and $\A$ a Banach operator ideal. Then, $E$ has the
$s\A$-approximation property if and only if  $\K_{\A}(F;E)=\K^{min}_{\A}(F;E)$, for all Banach
spaces $F$.
\end{corollary}

\begin{proof}
Since $\K_{\A}$ is a surjective ideal, it is left-accessible. Thus, by~\cite[Proposition~25.2]{DF},
$\K^{min}_{\A} = \overline{\mathcal F}\circ \K_{\A}$. An application of Proposition~\ref{PAyoper2}
completes the proof.
\end{proof}

The above corollary recovers the characterization of the $\kappa_p$-approximation property in
terms of the ideals $\K_p$ and $\K_p^{min}$, see the comments below \cite[Proposition 3.9]{GaLaTur}.
In general, the approximation property does not imply the $s\A$-approximation property. However, for
right-accessible ideals the claim is true as it happens with the  $\kappa_p$-approximation
property~\cite[Proposition~3.10]{GaLaTur}.

\begin{proposition} Let $E$ be a Banach space and $\A$ be a right-accessible Banach operator ideal.
If $E$ has the approximation property, then $E$ has the
$s\A$-approximation property.
\end{proposition}

\begin{proof}
If $E$ has the approximation property,  by~\cite[Proposition 25.11]{DF},
$$
(\K_\A^{min})^{sur}(F;E)=(\K_\A^{sur})^{min}(F;E)=\K_\A^{min}(F;E),
$$
for every Banach operator ideal $\A$ and for every Banach space $F$. Since $\A$ is
right-accessible, applying  Proposition~\ref{tot-acces} we see that $\K_\A(F;E)=\K_\A^{min}(F;E)$
for every Banach space $F$. Therefore, by the above corollary, $E$ has the $s\A$-approximation
property.
\end{proof}

\subsection{On the dual space $E'_{\A}$ and the $\epsilon$-product of Schwartz}

The classic approximation property has a well known reformulation
in terms of the $\epsilon$-product of Schwartz. More precisely, $E$
has the approximation property if and only if $E\otimes F$ is dense
in $\mathcal L_{\epsilon}(E'_{c};F)$ for every locally convex space
$F$,~\cite[Expos\'e~14]{Sch}. This result remains valid for the
$\A$-approximation property with a general Banach operator
ideal $\A$. We denote by $\mathcal
L_{\epsilon}(E'_{\A};F)$ the space of all linear continuous maps
from $E'_{\A}$ to  a locally convex space $F$, endowed with the topology of uniform
convergence on all equicontinuous sets of $E'$. The topology on  $\mathcal
L_{\epsilon}(E'_{\A};F)$ is
generated by the seminorms $\beta(T)=\sup_{x' \in B_{E'}}
\alpha(Tx')$, where $\alpha \in cs(F)$, the set of all
continuous seminorms on $F$. As usual, $U^\circ$
denotes the polar set of a set $U$.

\begin{theorem}
Let $E$ be a Banach space and $\A$ a Banach operator ideal. The following statements are equivalent.
\begin{enumerate}[\upshape (i)]
\item $E$ has the $\A$-approximation property.
\item $E\otimes F$ is dense in $\mathcal L_{\epsilon}(E'_{\A};F)$ for all
locally convex spaces $F$.
\item $E\otimes E'$ is dense in $\mathcal L_{\epsilon}(E'_{\A};E'_{\A})$.
\end{enumerate}
\end{theorem}

\begin{proof}
Suppose (i) holds. Fix  a locally convex space $F$,  a continuous
seminorm $\beta$ so that $\beta(T)=\sup_{x' \in B_{E '}}
\alpha(Tx')$, where $\alpha \in cs(F)$ and  $\ep >0$. Take $T \in
\mathcal L_{\epsilon}(E'_{\A};F)$, then we may find an absolutely 
convex $\A$-compact set $K \subset
E$ such that
$\sup_{x' \in K^\circ} \alpha(T(x'))\leq 1$. If $U=\{y \in F \colon
\alpha(y)\leq 1\}$, then $U$ is a neighborhood of $F$ and
$T'(U^\circ)\subset K$. Since $E$ has the $\A$-approximation property, there exists a
finite rank operator $S$ such that $\|Sx-x\| <\ep$ for all $x \in
K$. In particular, $\|Sx-x\| <\ep$ for all $x \in T'(U^\circ)$.
Then, for all $x'\in B_{E'}$ and $y' \in U^\circ$ we have
$|x'(S(T'(y'))-T'(y'))|<\ep$. Now, suppose that $S=\sum_{j=1}^{n}
x'_j \underline{\otimes}x_j$ with $x_1,\ldots,x_n \in E$,
$x'_1,\ldots,x'_n \in E'$, then
$$
|y'\big(\sum_{j=1}^{n}x'(x_j) T(x'_j) - T(x')\big)|<\ep,
$$
for  any $y' \in U^\circ$ and  $x'\in B_{E'}$. Therefore, taking
$R=\sum_{j=1}^{n} x_j \underline{\otimes}T(x'_j)$ we have
$\beta(R-T)< \ep$, which proves (ii).

That (ii) implies (iii) is clear. To finish the proof suppose that
(iii) holds. Take an absolutely convex $\A$-compact set $K\subset E$
and let $\beta$ be the continuous seminorm given by
$\beta(T)=\sup_{x' \in B_{E'}}\alpha(T)$, where $\alpha$ is the
Minkowski functional of $K^\circ$. Since $Id \in \mathcal L_{\epsilon}(E'_{\A};E'_{\A})$,
given $\ep >0$, there exist $S\in E\otimes E'$ such that $\beta(S-Id)< \ep$. Then, as above, 
$|x'(S'(x)-x)|<\ep$, for all $x' \in B_{E'}$ and $x \in K$.
Therefore, $\|S'(x)-x\|<\ep$ for all $x \in K$, and the proof
is complete.
\end{proof}

A Banach space $E$ has the classic approximation
property if and only if $E'_c$ has the approximation property, \cite[Expos\'e 14]{Sch}. Aron,
Maestre and Rueda show the analogous result for the $p$-approximation property \cite[Theorem
4.6]{AMR}. Here, we present a generalization of these results.

\begin{proposition} Let $E$ be a Banach space and $\A$ a Banach operator ideal. Then, $E$ has the
$\A$-approximation property if and only if $E'_\A$ has the approximation property.
\end{proposition}

\begin{proof}
The locally convex space $E'_\A$ has the approximation property if and only if for any  $\ep>0$, $K
\subset E$ an $\A$-compact set and $M\subset E'_\A$ relatively compact, there exists $S\in
\mathcal F(E;E)$ such that
\begin{equation}\label{E'A and AP}
|(S'-Id)(x')(x)|\leq \ep,\qquad \mbox{for all } x'\in M,\ x\in K.
\end{equation}

The continuity of the identity map $E'_c\rightarrow E'_\A \rightarrow (E',w^*)$ says that
relatively compact sets in $E'_\A$ coincide with $\|.\|$-bounded sets. Then,
$E'_\A$ has the approximation property if and only if in \eqref{E'A and AP}, $M$ is replaced by
$B_{E'}$, which is equivalent to say that $E$ has the $\A$-approximation property.
% have $\| (S-Id_E)(x)\|\le\ep$,  for all $x\in K$, for all $\ep>0$
% and all $\A$-compact set $K$. That is, $E$ has the $\A$-approximation property.
\end{proof}

Last but not least, we give a characterization
of an $\A$-compact operator in terms of the continuity and
compactness of its adjoint. This result is well known for compact
operators and was studied in the polynomial and holomorphic setting
in \cite{AS}. Recall that if $E$ and $F$ are locally convex spaces
and $U\subset F$ is absolutely convex and closed, for a linear
mapping  $T\colon E\rightarrow F$ we have $T(x) \in U$ if and only
if $|y'(Tx)|\leq 1$ for all $y' \in U^\circ$, with $x\in E$.

\begin{proposition}
Let $E$ and $F$ be Banach spaces,  $T\in \mathcal L(E;F)$ and  $\A$ a Banach operator ideal. The
following
statements are equivalent.
\begin{enumerate}[\upshape (i)]
\item $T \in \K_{\A}(E;F)$.
\item $T'\colon F'_{\A}\rightarrow E'$ is continuous.
\item $T'\colon F'_{\A}\rightarrow E'_c$ is compact.
\item $T'\colon F'_{\A}\rightarrow E'_{\mathcal B}$ is compact for any Banach
operator ideal $\mathcal B$.
\item There exists a Banach operator ideal $\mathcal B$ such that $T'\colon
F'_{\A}\rightarrow E'_{\mathcal B}$ is compact.
\item $T'\colon F'_{\A}\rightarrow E'_{w^*}$ is compact.
\end{enumerate}
\end{proposition}

\begin{proof} Suppose (i) holds, then $\overline{T(B_E)}=K$ is
$\A$-compact and $K^\circ$ is a neighborhood in $F'_{\A}$.
Thus, for $y' \in K^\circ$ we have that $\|T'(y')\|=\sup_{x \in
B_E}|T'(y')(x)|\leq 1$, proving (ii).

Now, suppose (ii) holds. Then, there exists a relative $\A$-compact
set $K\subset F$ such that $T'(K^\circ)$ is equicontinuous in $E'$
which, by the Ascoli theorem, is relatively compact in $E'_c$,
obtaining (iii). Since $Id\colon E'_c\rightarrow E'_{\mathcal B}$ is
continuous for any Banach operator ideal $\mathcal B$, (iii) implies
(iv). That (iv) implies (v) and (v) implies (vi) is clear. It remains to show that (vi) implies (i).
Let $L \subset E'$ be a $w^*$-compact set (hence $\|.\|$-bounded), and $K\subset F$ an absolutely
convex and $\A$-compact set such that $T'(K^\circ)\subset L$. As $T'(K^\circ)$ is
$\|.\|$-bounded, there exists $c >0$ such that $|T'(y')(x)| \leq c$ for every
$y' \in K^\circ$ and $x \in B_E$. Therefore,  $T(B_E)\subset cK$  which ends the
proof.
\end{proof}

\subsection{On the dual $E'$ and approximation properties given by $\A$}

We present a short discussion on approximation
properties for  dual spaces and the dual ideal of $\K_\A$. Recall  that, even though
$E$ has the approximation property, the
space of finite rank operators $\mathcal F(E;F)$ may not be dense in
the space of compact operators $\K(E;F)$. For this to happen it is required $E'$ with the
approximation property. The relations between density of
finite rank operators and approximation properties related operator ideals is not without
precedent. See~\cite[Theorem~2.8]{DOPS} for the $\mathcal N^p$-approximation property
and~\cite[Theorem~2.3]{DPS_dens} for the $s\mathcal N^p$-approximation property. To give our result,
involving a general operator ideal $\A$, we need the following lemma.

\begin{lemma}\label{lemaE'}
Let $E$ and $F$ be Banach spaces and $\A$ a Banach operator ideal.
\begin{enumerate}[\upshape (a)]
\item The set $E\otimes F$ is  $\tau_{s\A}$-dense in $\mathcal F(E';F)$.
\item The set $E\otimes F$ is $\tau_{\A}$-dense in $\mathcal F(E';F)$.
\end{enumerate}
\end{lemma}

\begin{proof}
Statement (b) follows from (a). To prove the first claim, let $T\in \mathcal F(E';F)$ and suppose
$T=\sum_{j=1}^n x''_j \underline{\otimes} y_j$ with $x_j''\in E''$, $y_j\in F$, $j=1,\ldots, n$ and
$\sum_{j=1}^n\|y_j\|\le 1$. Fix $\ep>0$ and $K\subset E'$ an $\A$-compact
set (and hence compact). By the Alaoglu theorem,  there exist
$x_1,\ldots,x_n \in E$ such that $\sup_{x' \in K} |x''_j(x') - x'(x_j)|<\ep$, for $j=1,\ldots,n$.
Take $S$ the finite rank operator, $S=\sum_{j=1}^n x_j \underline{\otimes} y_j$,
then
$$
\begin{array}{rl}
\m_{\A}((T-S)(K);F) &\le \sum_{j=1}^n \m_{\A}((x''_j-x_j)\underline{\otimes}y_j(K);F)\\
&= \sum_{j=1}^n\sup_{x' \in K} |x''_j(x') - x'(x_j)|\|y_j\|\\
&\le \ep,
\end{array}
$$
and the proof is complete.
\end{proof}

The $\A$ and the $s\A$-approximation properties  are  determined on dual spaces by
denseness of finite rank operators in the dual ideal of $\K_\A$, which is not a surprise at the
light of the examples mentioned above, \cite[Theorem~2.8]{DOPS}
and~\cite[Theorem~2.3]{DPS_dens}. We only give a proof for the result concerning the
$s\A$-approximation property, so we state this result first.

\begin{proposition} Let $E$ be a Banach space  and $\A$ a Banach operator ideal. The following are
equivalent.
\begin{enumerate}[\upshape (i)]
\item $E'$ has the $s\A$-approximation property.
\item For every Banach space $F$, $\mathcal F(E;F)$ is
$\|\cdot\|_{\K_{\A}^d}$- dense in $\K_{\A}^d(E;F)$.
\end{enumerate}
\end{proposition}

\begin{proof}
If (i) holds, fix $\ep>0$ and take $T \in \K_{\A}^d(E;F)$. Since
$T' \in \K_{\A}(F';E')$ and $E'$ has the $s\A$-approximation
property, by Lemma~\ref{lemaE'}, there exists $S \in \mathcal
F(E;F)$ such that $\|T-S\|_{\K_{ A}^d}= \|S'-T'\|_{\K_{\A}}\leq \ep$
which gives (ii).

Now suppose (ii) holds and take $T \in \K_{\A}(F;E')$. By
Proposition~\ref{K_A dd}, $T' \in \K_{\A}^d(E'';F')$ and  $T'\circ
J_E \in \K_{\A}^d(E;F')$. Fix $\ep>0$, by hypothesis, there
exists $S \in \mathcal F(E;F')$ such that $\|S-T'\circ J_E\|_{\K_{\A}^d}\leq \ep$. Since
$T= (J_E)'\circ T''\circ J_F$ and $S' \circ J_F \in \mathcal
F(F;E')$, we see that
$$
\|S' \circ J_F - T\|_{\K_{\A}} \leq \|S'-(T'\circ J_E)'\|_{\K_{\A}} =\|S-T'\circ
J_E\|_{\K_{\A}^d}\leq \ep.
$$
An application of Proposition~\ref{PAyoper2} gives the result.
\end{proof}

Finally, we have the analogous result for $E'$ and the $\A$-approximation property.

\begin{proposition} Let $E$ be a Banach space  and $\A$ a Banach operator ideal. The following are
equivalent.
\begin{enumerate}[\upshape (i)]
\item $E'$ has the $\A$-approximation property.
\item For every Banach space $F$, $\mathcal F(E;F)$ is $\|\cdot\|$- dense
in $\K_{\A}^d(E;F)$.
\end{enumerate}
\end{proposition}

\subsection*{Acknowledgements} The authors would like to thank Bernd Carl
for providing them with the reference \cite{CaSt} and to Ver\'onica
Dimant for her helpful comments.

\end{document}